\documentclass[preprint,3p]{elsarticle} %
\usepackage{amssymb}%
\usepackage{amsmath}%
\usepackage{amsthm}%
\usepackage{ulem}%
\newcommand{\op}[1]{\operatorname{#1}}%
\newcommand{\hil}{X} %
\newcommand{\strradtext}{1/(8r)}%
\newcommand{\radbig}{r} %
\newcommand{\radsmall}{s} %
\newcommand{\ball}{\overline B}
\newcommand{\quot}[1]{\frac{\delta_\Omega(#1)}{{#1}^2}}     %

\newcommand{\Quot}[1]{{\delta_\Omega(#1)}/{{#1}^2}}
\newcommand{\Quott}[2]{{\delta_{{#1}}(#2)}/{{#2}^2}}
\newcommand{\Quotto}[2]{{\delta^\circ_{{#1}}(#2)}/{{#2}^2}}
\newcommand{\eps}{\varepsilon}
\newcommand{\unitvector}{v \in \partial B(0,1)} 
\newcommand{\arc}[1]{A_{(#1)}}
\newcommand{\set}[2]{\{#1 \, | \, #2\} }
\newcommand{\fracquot}[1]{1/(#1)}
\newcommand{\sline}[1]{\{ \lambda #1 \, | \, \lambda \in \R\}}
\newcommand{\point}{p}
\newcommand{\indexx}{i}
\newcommand{\gll}{\begin{equation*}} 
\newcommand{\gle}{\end{equation*}}   
\newcommand{\R}{\mathbb{R}}
\newcommand{\N}{\mathbb{N}}
\newcommand{\fol}{\Rightarrow}
\newcommand{\tfae}{the following conditions are equivalent}
\theoremstyle{plain}
\newtheorem{Satz}{Theorem}
\numberwithin{Satz}{section}
\newtheorem{Cor}[Satz]{Corollary}
\newtheorem{Prop}[Satz]{Proposition}
\newtheorem{Lemma}[Satz]{Lemma}
\theoremstyle{definition}
\newtheorem*{Definition}{Definition}

\journal{Journal of Mathematical Analysis and Applications} %
\usepackage{gunther2e-elsevier}
\hypersetup{	pdftitle={Local characterization of strongly convex sets},%
		pdfauthor={Alexander Weber,} {Gunther Rei\ss{}ig},%
		pdfkeywords={Strongly convex set,} {Modulus of convexity,} {Strong convexity,} {Convexity,} {Hilbert space}	} 
\begin{document}
\begin{frontmatter}
\title{Local characterization of strongly convex sets\tnoteref{version{},DFG{}}}
\tnotetext[version{}]{This is the accepted version of a paper published in \textit{Journal of Mathematical Analysis and Applications}, vol. 400, no. 2, 2013, pp. 743-750, \href{http://dx.doi.org/10.1016/j.jmaa.2012.10.071}{DOI:10.1016/j.jmaa.2012.10.071}}
{}\tnotetext[DFG{}]{This work has been supported by the German Research Foundation (DFG) under grant no. RE 1249/3-1.} %
\author[UniBwM]{Alexander~Weber\corref{cor1}}
\ead{A.Weber@unibw.de}
\author[UniBwM]{Gunther~Rei\ss ig} 
\ead[url]{http://www.reiszig.de/gunther/} 
\address[UniBwM]{Department of Aerospace Engineering, Institute of Control Engineering, Universit\"at der Bundeswehr M\"unchen, Werner-Heisenberg-Weg 39, 85577 Neubiberg, Germany}
{}\cortext[cor1]{Corresponding author}
\begin{abstract}
Strongly convex sets in Hilbert spaces are characterized by local %
properties. One quantity which is used for this purpose is a %
generalization of the modulus of convexity $\delta_\Omega$ %
of a set $\Omega$. We also show that $\lim_{\eps \to 0} \Quot{\eps}$ exists whenever $\Omega$ is closed and convex.
\end{abstract}
\end{frontmatter}
\makeatletter
\def\refstepcounter#1{\stepcounter{#1}%
    \protected@edef\@currentlabel
       {\csname p@#1\endcsname{\csname the#1\endcsname}}%
}
\makeatother
\section{Introduction}
\label{s:intro}
Tietze and Nakajima were the first to obtain local characterizations of the convexity of subsets of a finite-dimensional Euclidean space \cite{Tietze1928,Tietze1929,Nakajima1928}. Their results have been extended to infinite dimensions as well as to various generalizations of the notion of convexity; see \cite{Valentine64,VanDeVel93,PolovinkinBalashov04} and the references therein. In the present paper, we present a novel local characterization of one such generalization in a Hilbert space setting in which the role of line segments in classical convexity is assumed by \textit{lenses} \cite{Blanc43}, as detailed below. 

Here and throughout, $X$ denotes a real Hilbert space of dimension at
least 2 endowed with inner product $\innerProd{\cdot}{\cdot}$ and norm
$\| \cdot \|$, and $\ball(c,r)$ is the closed ball of radius $r$
centered at $c \in X$, where we adopt the convention that $\ball(c,0)=\{c\}$.
\begin{Definition}
Let $r>0$. A set $\Omega \subseteq X$ is called \begriff{$r$-convex} if 
\begin{equation}
\bigcap_{x,y \in \ball(c,r)} \ball(c,r) \subseteq \Omega
\label{linse}
\end{equation}
for all $x,y \in \Omega$, and $\Omega$ is called \textit{strongly convex} if $\Omega$ is $s$-convex for some $s>0$. The intersection in \ref{linse} is taken over
all balls containing $x$ and $y$ and is considered equal to $X$ if no
such ball exists.
\end{Definition}
Early results of Mayer %
imply the following analogues of the above-mentioned results of Tietze %
and Nakajima \cite{Mayer35}; see Section \ref{s:MayersTheorems}.
\begin{Satz}
\label{rconvexweaklyrconvex}
Let $r > 0$ and $\Omega \subseteq X$ be closed and connected. Then the
set $\Omega$ is $r$-convex if and only if it is locally $r$-convex, by which we
mean that for each point $x \in \Omega$ there exists a neighborhood
$U \subseteq X$ of $x$ such that $U \cap \Omega$ is $r$-convex.
\end{Satz}
\begin{Satz}
\label{rconvsphersupp}
Let $r > 0$ and $\Omega \subseteq X$ be open and connected. Then the
set $\Omega$ is $r$-convex if and only if it is spherically supported with radius
$r$ at each of its boundary points locally, by which we mean that for
each boundary point $x \in \partial \Omega$ there exists a
neighborhood $U \subseteq X$ of $x$ and some $v \in X$, $\| v \| = 1$,
such that $U \cap \Omega \subseteq \ball(x - r v,r)$.
\end{Satz}
The following related result has recently been established by Balashov
and Repov{\v{s}} \cite{BalashovRepovs11b}:
A closed, convex and bounded subset $\Omega$ of the Hilbert space $X$
is $r$-convex, $r > 0$, if and only if there exists some $\varepsilon > 0$ such
that condition \ref{linse} holds for all points $x,y \in \Omega$
whose distance does not exceed $\varepsilon$.
Moreover, the property of strong convexity of a set $\Omega$ has also
been characterized in terms of its \textit{modulus of convexity}
$\delta_\Omega$ \cite{Polyak66},
\begin{equation}
\label{e:deltaOmega}
\delta_\Omega( \varepsilon )
=
\sup
\big \{
\delta \geq 0
\; \big |\;
x, y \in \Omega, \|x - y\| = \varepsilon
\Longrightarrow
\ball((x+y)/2,\delta) \subseteq \Omega
\big \},
\ \ \
0 \leq \varepsilon < \infty.
\end{equation}
Specifically, if $\Omega \subseteq X$ is a non-empty, closed and convex set for which the limit
\begin{equation}
\label{e:BR11Limit}
\lim_{\genfrac{}{}{0pt}{2}{\varepsilon \to 0}{\eps>0}}
\delta_{\Omega}(\varepsilon)
/
\varepsilon^2
\end{equation}
exists and is positive, then $\Omega$ is $r$-convex, $r>0$, if and only if the
limit is not less than $1/(8r)$
\cite{BalashovRepovs11b}.
These results allow us to test the strong convexity of a set by
verification of conditions on points $x$ and $y$ that are sufficiently
close. Still, the former characterization is not purely local since
the value of $\varepsilon$ is required to be independent of the
location of the points $x$ and $y$ in $\Omega$. (Here, we call a characterization \begriff{local} if it takes the following form: For all $x \in \Omega$ there exists a neighborhood $U$ of $x$ such that $U \cap \Omega$ satisfies a certain condition.) Analogously, the
limit \ref{e:BR11Limit} is required to show a kind of uniformity which, in
view of Theorem \ref{rconvexweaklyrconvex}, seems to be unnecessarily
restrictive. Indeed, the existence of the limit \ref{e:BR11Limit} is
open even for strongly convex sets $\Omega$.

In the present paper, we extend the results of Balashov and Repov{\v{s}} in that we establish a local characterization of
the property of strong convexity in terms of a
suitable generalization of the modulus \ref{e:deltaOmega}.
We also show that the limit \ref{e:BR11Limit} exists whenever $\Omega$ is closed and convex. 

Strongly convex sets appear several times in applications; see \cite{Plis75,IvanovPolovinkin95,BalashovRepovs09} and the references therein.
Our research has been particularly motivated by the fact that local characterizations of strong convexity are indispensable for strengthening recent results in control theory which concern geometric properties of image sets of time-1 maps of ordinary differential equations \cite{i07Convex,i11abs}.
\section{The main results}
\label{s:results}
The statement of our results will involve the following generalization of the modulus \ref{e:deltaOmega}.
\begin{Definition}
Let $\Omega \subseteq \hil$ and $x,y \in \Omega$. The formula
\begin{align*}
\delta_{\Omega,x}^\circ(\eps)
& :=
\sup \big \{ \delta \geq 0 \; \big | \; y \in \Omega, \
v \in \hil, \ \|v\| =1, 
\  \|x-y\|=\eps, \ \innerProd{v}{x-y} = 0  \Longrightarrow {(x+y)}/{2} +\delta \cdot v \in \Omega \big \}
\end{align*} 
defines a map $\delta_{\Omega,x}^\circ$ on the nonnegative real numbers, where
we have adopted the convention that $\sup \emptyset =- \infty$.
\end{Definition} 
The following are the main results of this paper.
\begin{Satz}
The limit \ref{e:BR11Limit} exists (in $\R$) for every closed and convex subset $\Omega \subseteq X$, $\Omega \neq X$.
\label{existlim}
\end{Satz}
\begin{Satz}
Let $r>0$ and $\Omega \subseteq \hil$ be closed and connected. Then \tfae:
\begin{asparaenum}[(i)]
\item
\label{localmainthm:i}
{$\Omega$ is $r$-convex.}
\item
\label{localmainthm:ii}
{$\lim_{\varepsilon \to 0,\eps>0} \Quot{\eps} \geq \strradtext$.}
\item
\label{localmainthm:iv}
{$\Omega$ is convex and $\liminf_{\varepsilon \to 0,\eps>0} \Quotto{\Omega,x}{\varepsilon} \geq \strradtext$ for all $x \in \Omega$.}
\end{asparaenum}
\label{localmainthm}
\end{Satz}
Given the existence of the limit in \ref{localmainthm:ii}, the latter
condition is the characterization of Balashov and Repov{\v{s}}
\cite{BalashovRepovs11b}. 
%
We remark that the assumption of convexity of $\Omega$ in condition \ref{localmainthm:iv}, which is purely local, cannot be omitted, as the example of two closed balls intersecting at exactly
one point shows.

In the proofs of Theorems \ref {existlim} and \ref{localmainthm} given in Section \ref{s:ProfOfMainResults} we will
make use of auxiliary results to be established in Sections
\ref{s:prelims}, \ref{s:MayersTheorems} and
\ref{s:ProfOfMainResults}. In particular, simple proofs of Theorems
\ref{rconvexweaklyrconvex} and \ref{rconvsphersupp} of Mayer
in Hilbert space are given in Section \ref{s:MayersTheorems}, and the following characterizations from
  Section \ref{s:ProfOfMainResults} may be of independent interest.
\begin{Prop}
Let $r>0$ and $\Omega \subseteq \hil$ be closed and connected.
Then $\Omega$ is $r$-convex if and only if the following holds.
\begin{asparaenum}[(A)]
\item
\label{eq:sequence:arcprop}
{
There exists a sequence $(\eps_i)_{i \in \N}$
of positive real numbers converging to 0 such that for $x,y \in \Omega$ the boundary of the set on the left hand side of \ref{linse}
is contained in $\Omega$ whenever $\| x-y\| = \eps_i$ for some $i \in \N$.
}
\end{asparaenum}
\label{sequencearcprop}
\end{Prop}
We note that assuming property \ref{eq:sequence:arcprop}, it is not obvious that the set 
$\Omega$ is convex. Even if one strengthens \ref{eq:sequence:arcprop} by requiring the whole lens 
to be contained in $\Omega$, instead of merely its boundary, convexity of $\Omega$ is not evident. Proposition \ref{sequencearcprop} will turn out to be essential for the proofs of Theorems \ref{existlim} and \ref{localmainthm}. 

Proposition \ref{sequencearcprop} remains true in the limit $r\to \infty$ by which we mean the following result.
\begin{Prop}
Let $\Omega \subseteq X$ be closed and connected. Then $\Omega$ is convex if and only if the following holds.
\begin{asparaenum}[(A)]
\setcounter{enumi}{2}
\item
\label{eq:sequence:conv} 
 {
There exists a sequence $(\varepsilon_i)_{i \in \N}$ of positive real numbers converging to $0$ such that $\conv\{x,y\}\subseteq \Omega$ whenever $x,y \in \Omega$ and $\|x-y\|=\eps_i$ for some $i \in \N$. 
}
\end{asparaenum}
\label{GuntherLemma}
\end{Prop}
\section{Preliminaries}
\label{s:prelims}
Throughout the paper, $\R$ denotes the field of real numbers, and
$\N$, the set of natural numbers, $\N=\{1,2,\ldots \}$. As
usual, $[\sigma,\tau]$, $(\sigma,\tau)$, $(\sigma,\tau]$ and
$[\sigma,\tau)$ denote the closed, open and half-open, respectively,
intervals in $\R$ with end points $\sigma$ and $\tau$. The open ball of
radius $r >0$ and center $c \in X$ is denoted as $B(c,r)$. The closure,
the interior, the boundary and the convex hull of a set $\Omega \subseteq X$ are
denoted by $\op{cl}\Omega$, $\op{int}\Omega$, $\boundary{\Omega}$, %
and $\op{conv} \Omega$,
respectively. In particular, $\op{conv}\{x,y\}$ is the line segment
$\{tx+(1-t)y \, | \, t \in [0,1]\}$.

We will frequently use that $\delta_{\overline B(0,r)}(\eps) = r- \sqrt{r^2- \eps^2/4}$ for any $\eps \in [0,2r]$ and $r>0$ \cite[Chapter 3, Section 3]{Diestel75}. From this we also conclude that $\lim_{\eps \to 0, \eps>0} \Quott{\ball(0,r)}{\eps}= 1/(8r)$ and $\Quott{\ball(0,r)}{\eps}> 1/(8r)$ for all $r>0$ and $\eps \in (0,2r]$.

The definition of strong convexity which we have
  adopted in Section \ref{s:intro} has the appeal of being concise and
  completely analogous to the case for ordinary convexity. However, it
  is often more convenient to work with \begriff{short arcs}
  \cite{Mayer35} and related concepts, which are introduced below.
\begin{Definition}
Let $S$ be the intersection of a sphere $\partial B(c,r)$ of radius $r>0$ and a two-dimensional affine subspace of $X$ containing $c$. If $x,y\in S$ and $x \neq y$, then $S \setminus \{x,y\}$ consists of two connected components $A'$ and $A''$. If one component is not longer than the other, say $A'$ is not longer than $A''$, then $\closure A'$ is called a \textit{short arc of radius $r$} joining $x$ and $y$. The set $\{x\}$ for $x \in X$ is also called a \textit{short arc of radius $r$}. For $x,y \in \hil$ and $r>0$ we define the \textit{midpoint} $z$ of a short arc  $A$ of radius $r$ joining $x$ and $y$ as follows: $z \in A$ is the unique element given by $\innerProd{x-y}{z-(x+y)/2}=0$.
\end{Definition}
\begin{Definition}
Let $\Omega,\Omega' \subseteq X$ and $r>0$. A pair $(x,y) \in \Omega \times \Omega$ possesses the \textit{arc property for $\Omega'$ and (radius) $r$} if for every short arc $A$ of radius $r$ joining $x$ and $y$, it holds that $A \subseteq \Omega'$. The set $\Omega$ possesses the \textit{arc property for (radius) $r$} if every pair $(x,y) \in \Omega \times \Omega$ possesses the arc property for $\Omega$ and $r$.
\end{Definition}
We remark that the arc property for radius $r$ does not impose %
any restriction on pairs of points $x,y \in \Omega$ whose distance exceeds $2r$. %
Interestingly, if such points exist, then necessarily $\Omega = \hil$: 
\begin{Prop}
\label{Reissigprop}
Let $r>0$ and $\Omega \subseteq \hil$. Then $\Omega$ is $r$-convex if and only if it is connected and possesses the arc property for radius $r$.
\end{Prop}
Under the additional assumption that $\Omega$ be closed, this result could be obtained from Theorem \ref{rconvexweaklyrconvex} of Mayer, or alternatively, from a combination of the theorem of 
Tietze and Nakajima \cite[Theorem 4.4]{Valentine64} and two results of Balashov and Repov{\v{s}} \cite[Theorem 2.1]{BalashovRepovs09}, \cite[Lemma 3.2]{BalashovRepovs11b}. We will give a direct and elementary proof below which does not rely on the closedness of $\Omega$. Proposition \ref{Reissigprop} will then be used in 
Section \ref{s:MayersTheorems} to establish Mayer's theorems in Hilbert space.

The following term will be used below: Two points $x$ and $y$ of a subset $\Omega \subseteq X$ are
\textit{polygonally connected} if there exist $p_1, \ldots, p_{n-1} \in \Omega$, $n \in \N$, such that
$\bigcup_{i=0}^{n-1} \conv\{p_i,p_{i+1}\} \subseteq \Omega$, where
$p_0 = x$ and $p_n = y$.
\begin{proof}[Proof of Proposition \ref{Reissigprop}]
Without loss of generality, we assume $r = 1$. The necessity of the condition is
easily established, so we prove only its sufficiency. Furthermore, if
$x,y \in \Omega$ such that $\|x-y\| \leq 2$ then it is straightforward
to show that \ref{linse} holds. In the remainder of the proof
we will make frequent use of the latter without mentioning it
explicitly. Suppose now that there exist $x,y \in
\Omega$ such that $\|x-y\| > 2$. We claim that in this case $\Omega =
\hil$. Indeed, as $\Omega$ is locally convex \cite[Theorem H 1]{Mayer35}, it
is polygonally connected \cite[Theorem 4.3]{Valentine64}. Hence, without
loss of generality, there exists $z \in \Omega$ such that $\Delta \subseteq \Omega$
where $\Delta = \conv\{x,z\} \cup \conv\{y,z\}$. Consequently, there exist $a$ and $c$ in the relative interior of the line segment $\conv\{x,z\}$ and $\conv\{y,z\}$, respectively, such that $\|a-c\|=2$. By the arc property, there exists $a' \in \Omega$ such that $a \in \conv\{a',c\} \subseteq \Omega$ and $\|a'-c\|>2$. Without loss of generality, we assume that $\hil$ is the Euclidean plane
 $\mathbb{R}^2$ and  $p_{+1} := a' = (1+\gamma,0)$,
  $p_{-1} := c = ( -(1+ \gamma),0)$
for some $\gamma \in (0,1/2)$. We consider the points
\begin{equation}
\label{p:point:coord}
p_{\pm n} =\Big  (\pm (1+ \gamma \cdot n^{-{1}/{2}}), \gamma^2 \textstyle \sum_{k=2}^n k^{-1} \Big ), \ n \in \N,
\end{equation}
and claim that $\conv\{p_n,p_{-n},p_{n+1},p_{-(n+1)}\} \subseteq \Omega$ for all $n \in \N$. Indeed, if $\conv\{p_n,p_{-n}\} \subseteq \Omega$ for some $n \in \N$, then $\ball(p_{\pm n} \mp (1,0),1) \subseteq \Omega$ since $\Omega$ is $1$-convex. This implies $p_{n+1} \in \Omega$ since $$\|p_{n+1} - (p_n-(1,0))\|^2 = (1+(n+1)^{-1/2}\gamma - n^{-1/2}\gamma)^2+(n+1)^{-2}\gamma^4<1$$ using \ref{p:point:coord}, and analogously for $p_{-(n+1)}$. Moreover, if $q_n = (p_n+p_{-n})/2$, then $q_n,q_{n+1} \in \ball(p_{\pm n} \mp (1,0),1)$ as $$1 > \gamma^2/n + \gamma^4/(n+1)^2=\|q_{n+1}-(p_{\pm n} \mp (1,0))\|^2>\|q_n - (p_{\pm n}\mp (1,0))\|^2,$$ which proves our claim. It follows that $\{0\} \times \R \subseteq \Omega$ as the series diverges and the previous arguments also apply with $-p_{\pm n}$ in place of $p_{\pm n}$, and suitable modifications of our choice of $p_{\pm 1}$ yield $\Omega = \R^2$.
\end{proof}
\begin{Definition}
Let $\Omega \subseteq \hil$ and $x \in \boundary \Omega$. A vector $v \in \hil$ is \textit{normal to $\Omega$ at $x$} if $\innerProd{v}{y-x} \leq 0$ for all $y \in \Omega$. If, additionally, $\|v\|= 1$ then $v$ is a \textit{unit normal to $\Omega$ at $x$}.
\end{Definition}
So every unit normal defines both a supporting half-space and a supporting hyperplane of $\Omega$ at $x$ and vice versa. For these well-known concepts see e.g. \cite{Valentine64}. The proof of the following result from \cite{Mayer35} carries over to the Hilbert space setting.
\begin{Lemma}
Let $r>0$ and let $\Omega \subseteq \hil$ be closed and convex. Let $x \in \partial \Omega$ and 
$y \in \Omega$ be such that $\|x-y\|\leq 2r$. If the pair $(x,y)$ possesses the arc property for $\Omega$ and $r$, %
then $y \in \ball(x-rv_x,r)$ for all unit normals $v_x$ to $\Omega$ at $x$. 
\label{LemmaMayer}
\end{Lemma} 
Most of the following characterizations of strong convexity have been already
observed in \cite[Proposition 3.1]{FrankowskaOlech81} for the finite-dimensional
case $X=\R^n$.
\begin{Prop}
\label{p:globalchar}
Let $r > 0$ and $\Omega \subseteq \hil$ be closed. Then the following conditions are equivalent: 
\begin{asparaenum}[(i)]
\item
\label{p:globalchar:i}
{$\Omega$ is $r$-convex.}
\item
\label{p:globalchar:ii}
{$\Omega$ is convex and $\Omega \subseteq \overline{B}(x-rv_x,r)$ for all $x \in \partial \Omega$ and all unit normals $v_x$ to $\Omega$ at $x$.}
\item
\label{p:globalchar:iv}
{$\Omega= \bigcap_{c \in M} \overline B(c,r)$ for some $M \subseteq \hil$.}
\item
\label{p:globalchar:v}
{The inequality
\gll
\delta_{\Omega}(\varepsilon) \geq r - \sqrt{r^2-{\varepsilon^2}/{4}}
\gle
is fulfilled for every $\varepsilon$.
}
\end{asparaenum}
\end{Prop}
\begin{proof}
Without loss of generality, let $\Omega$ be non-empty and $\Omega \neq \hil$. Implication \ref{p:globalchar:i} $\fol$ \ref{p:globalchar:ii} follows easily from Proposition \ref{Reissigprop} and Lemma \ref{LemmaMayer}. For the proof of implication \ref{p:globalchar:ii} $\fol$ \ref{p:globalchar:iv}, let $\Omega ' = \bigcap \overline B ( x-rv_x,r)$ where the intersection is over all $x \in \partial \Omega$ and all unit normals $v_x$ to $\Omega$ at $x$. By hypothesis, $\Omega \subseteq \Omega '$. If $y \in \Omega '\setminus \Omega$, then by \cite[3.1.2, Theorem 2 and Corollary 2]{Ioffe}, there exists a supporting half-space $H^-$ of $\Omega$ given by a unit normal $v_z$ to $\Omega$ at some $z \in \partial \Omega$ such that $y \notin H^-$. Since $\overline B(z-rv_z,r) \subseteq H^-$, it follows that $y \in H^-$ which is a contradiction. Implication \ref{p:globalchar:iv} $\fol$ \ref{p:globalchar:v} directly follows from the properties of the modulus $\delta_\Omega$ given at the beginning of this section, so \ref{p:globalchar:v} $\fol$ \ref{p:globalchar:i} is left to prove. To economize on notation, we denote by $A_{(x,y)}$ a short arc of radius $r$ joining two elements $x$ and $y$. Let $x,y \in \Omega$. By hypothesis $z \in \Omega$ whenever $z$ is the midpoint of a short arc $A:=A_{(x,y)}$ of radius $r$. Without loss of generality, we reduce our arguments to the two-dimensional Euclidean plane containing $\conv A$ and assume $x+y = 0$. Let $z$ be the midpoint of $A$. Given any point $a \in A$ we now construct a sequence $(p_n,q_n)_{n \in \N}$ on $A$ such that $\lim_{n\to \infty} p_n = \lim_{n \to \infty} q_n = a$. We define $q_1:=z$ and $p_1 \in \{x,y\}$ such that $a \in A_{(p_1,q_1)}$. Applying the same argument as before to $p_1$ and $q_1$ in place of $x$ and $y$ we determine $q_2 \in \Omega$ as the midpoint of the short arc $A_{(p_1,q_1)} \subseteq A$. Inductively, we determine $q_{n+1}$ as the midpoint of $A_{(p_{n},q_{n})} \subseteq A$, and $p_{n+1} \in \{p_n,q_n\}$ such that $a \in A_{(p_{n+1},q_{n+1})}$. Then $\|p_n-a\|^2 \leq \|p_n-p_{n+1}\|^2 \leq \|x-y\|^2/2^n$, so the sequence converges to $a$. As $\Omega$ is closed, we conclude that $A \subseteq \Omega$ and hence, Proposition \ref{Reissigprop} implies \ref{p:globalchar:i}. 
\end{proof}
\section{The theorems of Mayer in Hilbert space}
\label{s:MayersTheorems}
We give simple proofs of Theorems %
\ref{rconvexweaklyrconvex} and \ref{rconvsphersupp} in Hilbert space %
and we provide an auxiliary corollary to %
Theorem \ref{rconvsphersupp} in this section. %
The latter as well as Theorem \ref{rconvexweaklyrconvex} will be applied in 
Section \ref{s:ProfOfMainResults}.%
\begin{proof}[Proof of Theorem \ref{rconvexweaklyrconvex}]
It is clear that $r$-convexity implies local $r$-convexity. For the converse, first observe that $\Omega$ is convex by the analogous theorem of Tietze and Nakajima \cite[Theorem 4.4]{Valentine64} as it is locally convex. 
We will show below that $\Omega$ possesses the arc property for radius $r$. Assume there exists a short arc $A \subseteq \hil$ of radius $r$ joining $x,y \in \Omega$ such that $A \nsubseteq \Omega$. Let $x+y = 0$ without loss of generality, and let $\unitvector$ such that $\innerProd{v}{x-y} = 0$ and such that some scalar multiple of $v$ meets $A$. We define $\delta=\delta_{\ball (0,r)}(\|x-y\|)$ and
\gll
z = \arg \max \{ \| w\| \, | \, w \in \Omega \cap v\cdot[0,\delta]\}.
\gle 
If both $(x,z)$ and $(y,z)$ possess the arc property for $\Omega$
and $r$, then $z\notin A$, since $A\subseteq \Omega$
otherwise. However, $\|z\| < \delta$ implies that there exist $w \in
\sline{(x-y)}$ and $\mu > \|z\|$ such that $w+\mu v, -w+\mu v \in
\Omega$, which together with the convexity of $\Omega$ contradicts the
choice of $z$. Therefore, one of the pairs $(x,z)$ and $(y,z)$ does not possess the arc property for $\Omega$ and $r$. Hence, we may construct a sequence $(p_n,q_n)_{n \in \N} \in \Omega \times \Omega$ for which $\lim_{n \to \infty} p_n = \lim_{n \to \infty} q_n = a \in \Omega$ (as in the proof of Proposition \ref{p:globalchar}) and $(p_n,q_n)$ does not possess the arc property for any $n \in \N$. This contradicts the local $r$-convexity of $\Omega$. 
\end{proof} 
\begin{proof}[Proof of Theorem \ref{rconvsphersupp}]
First we note that $\Omega$ is convex by the theorem of Tietze
\cite[Theorem 4.10]{Valentine64}. It is enough to prove that every pair of points $x,y \in \Omega$ possesses the arc property for $\closure \Omega$ and radius $r$. Indeed, if the latter is true then $\closure \Omega$ is $r$-convex, and hence so is $\Omega$ as $\Omega = \interior( \closure \Omega)$. Let $x,y \in \Omega$, $x\neq y$, and let $A$ be a short arc of radius $r$ joining $x$ and $y$. Let $z$ be the midpoint of $A$. Consider the set 
\gll
\mathcal{T} =\set{t \in [0,1]}{A_t \subseteq \closure \Omega}
\gle
where $A_0 := \conv\{x,y\}$, and if $t>0$, then $A_t$ is the short arc joining $x$ and
$y$ such that $tz+(1-t)\cdot (x+y)/2 \in A_t$. As $\Omega$ is open we
conclude that $[0,t_0] \subseteq \mathcal{T}$ for some $t_0 >0$. Let $t \in [t_0,1) \cap \mathcal{T}$ and assume that there exists $w \in A_t$ such that $w \in \partial \Omega$. Then there exists a neighborhood $U$ of $w$ such that $U \cap \closure\Omega \subseteq \ball(w-rv,r)$ where $v$ is a unit normal to $\Omega$ at $w$. This is a contradiction to $A_t \subseteq \closure \Omega$ since the radius of $A_t$ exceeds $r$. %
Hence, $A_t \subseteq \Omega$, so $\mathcal{T}$ is open as $A_t$ is compact. Since $\mathcal{T}$ is closed by the closedness of $\closure \Omega$ the claim follows since $A = A_1 \subseteq \closure \Omega$. 
\end{proof}
We remark that Theorem \ref{rconvsphersupp} proves that conditions \ref{p:globalchar:i}-\ref{p:globalchar:v} of Proposition \ref{p:globalchar} are equivalent to:
\begin{asparaenum}[\itshape (i')]
\setcounter{enumi}{1}
\item
\label{p:globalchar:iii}
\textit{$\Omega$ is convex and $\Omega \subseteq \ball(x-rv_x,r)$ for all $x \in \partial \Omega$ and some unit normal $v_x$ to $\Omega$ at $x$.}
\end{asparaenum}
~

Theorem \ref{rconvsphersupp} particularly applies if $\Omega$ is the closure of an open, connected subset of $X$. Therefore we have the following consequence.
\begin{Cor}
Let $r>0$ and $\Omega \subseteq \hil$ be closed and convex. Assume that for every $x \in \partial \Omega$ there exists a neighborhood $U$ of $x$ such that for all $y \in U \cap \Omega$ the pair $(x,y)$ possesses the arc property for $U \cap \Omega$ and $r$. Then $\Omega$ is $r$-convex. 
\label{xarcporp}
\end{Cor}
\begin{proof}
Without loss of generality, $\Omega$ is neither empty nor a singleton. Then by hypothesis, $\interior \Omega \neq
\emptyset$, that is, $\Omega =\closure( \interior \Omega)$. Suppose
that $\Omega$ is not $r$-convex. Then $\Omega \neq \hil$ and there exists
$x \in \partial \Omega$ such that $U \cap \Omega \nsubseteq \ball
(x-rv,r)$ for any unit normal $v$ to $U \cap \Omega$ at $x$ and any
neighborhood $U$ of $x$ as in the hypothesis. Since $\Omega$ is convex
there exists a unit normal $v$ to $U \cap \Omega$ at $x$, so there exists
some $y \in U \cap \Omega$ such that $y \notin \ball(x-rv,r)$. By the
contrapositive of Lemma \ref{LemmaMayer}, the pair $(x,y)$ does not
possess the arc property for $U \cap \Omega$ and $r$, which is a contradiction.
\end{proof}
\section{Proof of the main results}
\label{s:ProfOfMainResults}
In this Section, we shall prove the main results presented in Section \ref{s:results}. Some auxiliary results are also proved.
To begin with, we first establish Proposition \ref{GuntherLemma}.
\begin{proof}[Proof of Proposition \ref{GuntherLemma}]
\newcounter{one}
\renewcommand{\point}{y_p}
\renewcommand{\indexx}{j}
The necessity of the condition is obvious. The proof of its sufficiency is divided into several steps. %
We first prove that if a pair of points $x,y \in \Omega$ is polygonally connected then %
$\conv\{x,y\} \subseteq \Omega$ (Claim 1 up to Claim 3). %
For this purpose it is enough to assume that there exists $z \in \Omega$ satisfying $\conv\{x,z\} \cup \conv\{y,z\} \subseteq \Omega$. Without loss of generality, let $x,y,z$ be not collinear, and $z=0$. Define $\Delta= \conv\{0,x,y\}$. We consider the set
\gll
\mathcal{T} = \{ t \in [0,1] \; | \; \conv\{tx,ty\} \subseteq \Omega \},
\gle
and reduce the remaining arguments to the two-dimensional plane that contains $\Delta$. 
\\
\refstepcounter{one}\label{claimatzero}\textsf{Claim \arabic{one}}: There exists $t_0 \in (0,1]$ such that $[0,t_0] \subseteq \mathcal{T}$. \\
Since $\|t(x-y)\|$ is %
monotonically increasing in $t$, it follows by
property \ref{eq:sequence:conv} that there exists $t_1 \in
\mathcal{T}$ such that, without loss of generality, $ \eps_1 = \|t_1(x-y)\|$ and
$\eps_1 \leq {\|x\|}/{2}$. We choose $t_0 \in (0,t_1]$ such that $\|t_0 y \| \leq \eps_1$. %
Let $p \in \conv\{0,t_0x,t_0y\}$. Obviously, there exists $t_p \leq t_0$ such that $p \in \conv\{t_px,t_py\}$. Consider the family of line segments $$\ell(\tau)=\set{(1-\lambda)p+\lambda [ (1-\tau)t_px + \tau x] }{\lambda \in \R} \cap \Delta, \ \tau \in [0,1].$$ The line segment $\ell(0)=\conv\{t_px,t_py\}$ has length less than or equal to $\varepsilon_1$. Next, let $\point$ be the unique point of $\ell(1) \cap \conv\{0,y\}$. Then $\ell(1) = \conv\{x,\point\}$ has length $\|x-\point\|$ which is not less than $\| x\| - \|\point\| \geq 2\eps_1 - \eps_1=\eps_1$ as $\|\point\| \leq \|t_py\|$. Since the length of the line segments $\ell(\tau)$ varies continuously as $\tau$ varies continuously, it follows that there exists a line segment of length $\eps_1$ with end points on $\conv\{0,x\} \cup \conv\{0,y\}$. Hence, $p \in \Omega$. \\
\refstepcounter{one}\label{claimopen}\textsf{Claim \arabic{one}}: If $t \in \mathcal{T}\setminus \{1\}$ then $\conv\{tx,ty\} \setminus \{tx,ty\} \subseteq \interior \Omega$.\renewcommand{\point}{p} \\
Let $\point = tx$, $q = ty$ and without loss of generality $t\geq t_0$, $\|\point-q\|=1$. Using Claim \ref{claimatzero} it follows that there exists $\indexx \in \N$ such that $\eps_j \leq 2/3$ and for $p_1=\point + \frac{\eps_\indexx}{2}(q-\point)$ it holds that $\conv\{\point,p_1\} \setminus \{\point\} \subseteq \interior \Omega$. Consider the points $p_n = \point + n \cdot \frac{\eps_\indexx}{2} (q-\point)$ for $n\in \N \cup \{0\}$ such that $p_n \in \Delta$. In particular, $p_3 \in \Delta$. If $p_{n-2} \in \interior \Omega$, $n \geq 3$, and if $U \subseteq \Omega$ is a neighborhood of $p_{n-2}$, then $\|p_n - p_{n-2}\| = \eps_{\indexx}$ and therefore
$$U \cap \boundary B(p_{n},\eps_\indexx) \subseteq \Omega.$$
So $\conv\{U \cap \boundary B(p_{n},\eps_\indexx),p_{n}\} \subseteq \Omega$, and %
thus inductively $\conv\{\point,p_{n-1}\} \setminus \{p\} \subseteq \interior \Omega$. %
By interchanging the role of $\point$ and $q$ the claim follows. \\ 
\refstepcounter{one}\label{claimTopen}\textsf{Claim \arabic{one}}: $\mathcal{T}=[0,1]$. \\
It is enough to prove that $\mathcal{T}$ is both closed and open in $[0,1]$.
Let $t \in \mathcal{T}$. By Claim \ref{claimatzero} we choose a neighborhood 
$U_{p} \subseteq \hil$ of $p=tx$ such that 
$U_{p} \cap \Delta \subseteq \Omega$. Interchanging the role of $tx$ and $ty$ 
and Claim \ref{claimopen} yield an open 
covering of $\conv\{tx,ty\}$ relative to the subspace topology on $\Delta$ so the compactness of $\conv\{tx,ty\}$ easily implies 
that $\mathcal{T}$ is open. $\mathcal{T}$ is closed as $\Omega$ is closed. \\
\refstepcounter{one}\label{claimconvex}\textsf{Claim \arabic{one}}: $\Omega$ is convex. \\
Let $p \in \Omega$. Without loss of generality, there exists one more point $q \in \Omega$ such that $\varepsilon_1 \leq \|p-q\|$. Since $\Omega$ is connected $\partial B(p,\varepsilon_1) \cap \Omega \neq \emptyset$, so $\|p-q\|=\varepsilon_1$ without loss of generality. Next, without loss of generality $\varepsilon_2 \leq \frac{\eps_1}{2}$. 
We choose $x \in \ball(p,\varepsilon_2/2)$. 
We consider the continuous map $f$ given by $f(t)= \|x-(tp+(1-t)q)\|$. %
Then $f(1) \leq \frac{1}2\varepsilon_2$ and $f(0) = \|x-q\| \geq \eps_1 - \frac{1}{2}\eps_2 \geq \frac{3}{2}\eps_2 > \varepsilon_2$. %
Hence, there exists $y \in \conv\{p,q\}$ such that $\|x-y\| = \varepsilon_2$. %
Thus, all pairs of points in $\ball(p,\varepsilon_2/2) \cap \Omega$ are polygonally connected in $\Omega$. %
So %
$\ball(p,\varepsilon_2/2) \cap \Omega$ is convex by Claim \ref{claimTopen}, and hence $\Omega$ is locally convex. Convexity of $\Omega$ follows therefore from the theorem of Tietze and Nakajima \cite[Theorem 4.4]{Valentine64}.
\end{proof}
As the reader may have noted, the proof of Proposition \ref{GuntherLemma} works also in real Banach spaces. So for a closed, connected and non-convex subset $\Omega$ of a real Banach space $E$ with norm $\|\cdot\|_E$ we conclude that
\gll
\inf\{ \varepsilon>0 \, | \, x, y \in \Omega, \|x-y\|_E = \varepsilon \Longrightarrow \conv\{x,y\} \subseteq \Omega\} > 0 .
\gle
The analogous result for strong convexity is Proposition %
\ref{sequencearcprop} which we will prove below.
\begin{proof}[Proof of Proposition \ref{sequencearcprop}]
Without loss of generality, let $\Omega$ consist of more than one point, and $\Omega \neq \hil$. The necessity of the condition is obvious; we prove its sufficiency. We first prove that $\Omega$ is convex. Let $x,y \in \Omega$ such that $\|x-y\|=\eps_i$ for some $i \in \N$. By Proposition \ref{GuntherLemma} it is enough to show that $\conv\{x,y\} \subseteq \Omega$. We consider a short arc $A$ of radius $r$ joining $x$ and $y$. Then $A\subseteq \Omega$, and without loss of generality, let $x+y = 0$. Let $z$ be the midpoint of $A$ and consider the set 
\renewcommand{\indexx}{j}
\gll
\mathcal{T} = \set{ t \in [0,1)}{ \ell(t) \subseteq \Omega },
\gle
where $\ell(t) = \set{(1-t)z + \lambda (x-y)}{\lambda \in \R} \cap
\op{conv} A$. Obviously, there exist $p_t,q_t \in A$ such that
$\conv\{p_t,q_t\} = \ell(t)$. Choose some $\eps_\indexx \leq \|x-z\|$
and first let $t<t_0:= \delta_{\ball(0,r)}(\eps_j)/\|z\|$. Consider the
family of short arcs $\arc{p,q}$ of radius $r$ joining points $p,q \in
A$ such that $\|p-q\|=\eps_j$ and $\arc{p,q} \nsubseteq A$. Hence, we
see that $\ell(t) \subseteq \Omega$, that is, $[0,t_0) \subseteq
\mathcal{T}$. Now let $t \in \mathcal{T} \cap [t_0,1)$. The family of
short arcs considered previously also shows that $p_t$ and $q_t$ are interior points of $\Omega \cap \conv A$ relative to the subspace topology on $\conv A$. Considering the family of short arcs of radius $r$ joining points $p',q' \in \ell(t)$ yields the implication that $\ell(t) \setminus \{p_t,q_t\} \subseteq \interior \Omega$. Since $\ell(t)$ is compact we conclude that $\mathcal{T}$ is open. As it is also closed by the closedness of $\Omega$, we have $\mathcal{T} = [0,1)$. As $\Omega$ is closed, it is convex. \\ \renewcommand{\indexx}{k}
\renewcommand{\point}{n} Finally, we prove that $\Omega$ is $r$-convex. Suppose there exist $x,y \in \Omega$ not possessing the arc property for $\Omega$ and $r$. To contradict this assumption, we reduce the arguments to the affine plane which contains $x,y$ and a short arc of radius $r$ joining $x$ and $y$ which is not contained in $\Omega$. As $\Omega$ is convex, it contains the points $x_\alpha = x + \alpha(y-x)$ and $y_\alpha=y-\alpha(y-x)$ for all $\alpha \in [0,{1}/{2}]$. We consider the non-empty set $$\mathcal{A}=\{ \alpha \in [0,1/2) \ | \ (x_\alpha,y_\alpha) \text{ possesses the arc property for } \Omega \text{ and } r \}.$$
$\mathcal{A}$ is closed by the closedness of $\Omega$. Defining $\beta = \inf \mathcal{A}$ it follows that $\beta \in \mathcal{A}$, and hence $\beta > 0$. \\
\textsf{Claim}: For any short arc $A$ of radius $r$ joining $x_\beta$ and $y_\beta$, it holds that $A \subseteq \interior \Omega$. \\ First, we note that $x_\beta$ is an interior point of $\Omega$ as follows: We choose a line segment contained in $\conv\{x,y\}$, containing $x_\beta$, having end points different from $x_\beta$ (this is possible since $\beta >0$) and having length $\eps_\nu$ for some $\nu \in \N$. By the hypothesis of property \ref{eq:sequence:arcprop} and as $\Omega$ is convex it easily follows that $x_\beta \in \interior \Omega$. By the same argument, $y_\beta \in \interior \Omega$. Without loss of generality, we assume $2\eps_1 \leq \|x_\beta - y_\beta \|$. For $p,q \in \Omega$ let $A_{(p,q)}$ denote a short arc of radius $r$ joining $p$ and $q$. For any $\point \in \N$, $n \geq 2$, let $p_0,\ldots,p_{\point} \in A=A_{(x_\beta,y_\beta)}$ such that $\|p_{\indexx} - p_{\indexx-2} \| = \varepsilon_1$ for $2\leq \indexx \leq \point$, $p_0 = x_\beta$, and let $p_1$ be the midpoint of $\arc{p_0,p_2}$ and $\|p_\point -y_\beta\| < \varepsilon_1$. If $p_{\indexx-2} \in \interior \Omega$ and if $U \subseteq \Omega$ is a neighborhood of $p_{\indexx-2}$, then it follows that $$\set{p \in \arc{u,p_\indexx}}{ \ u \in U \cap \boundary B(p_{\indexx},\eps_1)} \subseteq \Omega.$$
Thus, we have proved
\gll
p_{\indexx-2} \in \interior \Omega \ \Rightarrow \ \arc{p_{\indexx-2},p_{\indexx}} \setminus \{p_{\indexx}\} \subseteq \interior \Omega.
\gle
Inductively, it follows that $\arc{x_\beta, p_{n-1}} \subseteq \interior \Omega$. Interchanging the role of $x_\beta$ and $y_\beta$ we conclude that $A \subseteq \interior \Omega$. \\ Moreover, we conclude that $L\subseteq \op{int}\Omega$ where $L$ denotes the intersection on the left hand side of \ref{linse} with %
$x_\beta$, $y_\beta$ in place of $x$, $y$. We note that the set $L$ is compact as it is considered in a 2-dimensional plane. But the compactness of $L$ contradicts our definition of $\beta$.
\end{proof}
In the sequel we will need an easy but extremely useful remark, namely the following.
\begin{Lemma}
Let $\Omega \subseteq \hil$ be closed and let $(r_i)_{i \in \N}$ be a sequence of positive real numbers which converges to $r>0$. If $\Omega$ is $r_i$-convex for all $i \in \N$ then $\Omega$ is $r$-convex.
\label{rconv:sequence}
\end{Lemma}
\begin{Prop}
Let $r>0$ and $\Omega \subseteq \hil$ be closed and connected. Then $\Omega$ is $r$-convex if and only if 
\begin{equation}
\liminf_{{\eps>0,}{\varepsilon \to 0}}\limits \Quot{\varepsilon} \geq {1}/{(8r)}.
\label{limgl}
\end{equation}
\label{strongconvexityliminf}
\end{Prop}
\begin{proof}
If $\Omega$ is $r$-convex, then \ref{limgl} follows from Proposition \ref{p:globalchar} \ref{p:globalchar:v}. For the converse, if \ref{limgl} holds, note that for any $s>r$ and $s_0 := (r+s)/2$ there exists $\eps_0>0$ such that $\delta_\Omega(\varepsilon)\geq {\eps^2}/{(8s_0)}\geq s - \sqrt{s^2-\eps^2/4}$ for all $\eps \in [0,2\eps_0]$. Indeed, the remarks at the beginning of Section \ref{s:prelims} imply the last inequality. So for $x \in \Omega$ consider $U : = \ball(x,\eps_0)$. Then 
\gll
\delta_{U \cap \Omega}(\varepsilon) \geq s - \sqrt{s^2-\eps^2/4}.
\gle
Thus, $U \cap \Omega$ is $s$-convex by Proposition \ref{p:globalchar} \ref{p:globalchar:v}, and therefore $\Omega$ is $s$-convex by Theorem \ref{rconvexweaklyrconvex} for any $s>r$. Hence, $\Omega$ is $r$-convex by Lemma \ref{rconv:sequence}.
\end{proof}
\renewcommand{\radbig}{R}
\renewcommand{\radsmall}{r}
\begin{Prop}
Let $\radsmall>0$ and $\Omega \subseteq \hil$ be closed and convex and
assume the following.
\begin{asparaenum}[(A)]
\setcounter{enumi}{3}
\item
\label{eq:sequence:circle}
There exists a sequence $(\varepsilon_i)_{i \in \N}$ of positive real
numbers converging to $0$ such that for all $x \in \Omega$ and all $i \in \N$ it holds
$\delta^\circ_{\Omega,x}(\varepsilon_i) \geq
r-\sqrt{r^2-\varepsilon_i^2/4}$.
\end{asparaenum}
Then $\Omega$ is $\radsmall$-convex.
\label{thm:delta:circ}
\end{Prop}
To prove the previous statement we adopt the central idea of the proof of
Theorem 2.1 in \cite{BalashovRepovs11b} to our weakened hypotheses.
We begin with a weaker version of Proposition \ref{thm:delta:circ}.
\begin{Lemma}
Let $r>0$ and $\Omega \subseteq \hil$ be closed and convex and assume \ref{eq:sequence:circle}.
Then $\Omega$ is $2r$-convex.
\label{Lemma:delta:circ}
\end{Lemma}
\begin{proof}
Define $\delta_i:=r-\sqrt{r^2-\varepsilon_i^2/4}$. Let $\varepsilon > 0$ and let $i \in \N$ be sufficiently large that $\delta_i/\varepsilon_i < \eps$ and $\eps_i < 2r$. Let $x, y \in \Omega$, $\|x-y\| = \varepsilon_i$. Define $a := (x+y)/2 + \delta_i \cdot v$ for arbitrary $v \in \partial B(0,1)$ such that $\innerProd{v}{x-y}= 0$. By hypothesis, $a \in \Omega$. Without loss of generality, we reduce our arguments to the two-dimensional plane containing $\conv\{x,y,a\} \subseteq \Omega$, and assume $x+y=0$. Let $c$ denote the intersection of the line $\{\lambda a \, | \, \lambda \in \R \}$ with the line normal to $a-x$ through $x$. Then abbreviating $\mu := \eps_i/2$ we have $$\mu^2  = \delta_i \cdot \sqrt{\|c-x\|^2-\mu^2}$$ by Pythagorean theorem. So $\delta_i^2 \|c-x\|^2 =\mu^4 + \delta_i^2 \mu^2 = \mu^4\cdot(1+\delta_i^2/\mu^2).$
Hence,
\gll
\|c-x\| = \frac{\sqrt{1+(2\delta_i/\varepsilon_i)^2}}{4 \cdot \delta_i/\varepsilon_i^2} \leq \frac{\sqrt{1+4\eps^2}}{4 \cdot \delta_i/\varepsilon_i^2} \leq  {2r\cdot \sqrt{1+4 \eps^2}} =:s.
\gle
Hence, any short arc of radius $s$ joining $x$ and $y$ is contained in $\Omega$. Thus, $\Omega$ possesses property \ref{eq:sequence:arcprop} of Proposition \ref{sequencearcprop} with any radius $s > 2r$ at place of $r$. So $\Omega$ is $2r$-convex by Proposition \ref{sequencearcprop} and Lemma \ref{rconv:sequence} using the limit $\eps \to 0$.
\end{proof}
\renewcommand{\radbig}{R}
\renewcommand{\radsmall}{r}
\begin{proof}[Proof of Proposition \ref{thm:delta:circ}] As a first step $\Omega$ is proved to be $\frac{2\radbig}{1+\radbig / \radsmall}$-convex provided that $\Omega$ is $\radbig$-convex, $\radbig > \radsmall$. Without loss of generality, let $x,y \in \Omega$ such that $2r>\|x-y\| = \varepsilon_i=:\varepsilon$ for some $i \in \N$. Define $\delta:=\radsmall-\sqrt{\radsmall^2-\varepsilon^2/4}$ and $a := (x+y)/2 + \delta \cdot v$ for arbitrary %
$\unitvector$ such that $\innerProd{v}{x-y}= 0$. Property \ref{eq:sequence:circle} ensures $a \in \Omega$. Without loss of generality, we reduce our arguments to the two-dimensional plane containing $\conv\{x,y,a\} \subseteq \Omega$, and assume $x+y=0$. Define $e_1=-x$ and $e_2=a$. By hypothesis, $\conv\{A_{(x,a)},A_{(-x,a)}\} \subseteq \Omega$ where $A_{(\cdot)}$ denotes the short arc of radius $\radbig$ joining $a$ to $x$ and $-x$, respectively, such that the center of the corresponding sphere has second coordinate less than $0$. So $A_{(x,a)}$ is a subset of $\partial B(c',\radbig)$ where $c'$ has positive first coordinate since $\radbig>\radsmall$. Denote by $c$ the unique point of $\{\lambda e_2 | \lambda \in \R\} \cap \{ x+ \lambda (x-c') | \lambda \in \R \}$, and $\rho:=\|x-c\|$. Let $z=a-x$. Let $l_x$ and $l_a$ be the lines tangent to $B(c',\radbig)$ passing through $x$ and $a$, respectively. Choose $w\in \hil$ such that $\innerProd{w}{e_2}>0$ and $l_x = \{ x+\lambda w | \lambda \in \R\}$. For the angle between $e_1$ and $z$ we write $\alpha$, and for the angle between $z$ and $w$ we write $\beta$. Then, we note that the angle between $-e_2$ and $c'-a$ equals $\alpha - \beta$. The angle between $e_2$ and $x-c$ equals $\alpha + \beta$. By hypothesis, $\sin(\alpha \pm \beta) \neq 0$. Applying the sine law to the triangle with vertices $c$, $c'$ and $a$, we have
\gll
\frac{\radbig-\rho}{\sin({\alpha - \beta})} = \frac{\radbig}{\sin({\pi - (\alpha + \beta)})} = \frac{\radbig}{\sin(\alpha + \beta)}.
\gle
Thus,
\begin{equation}
\frac{\radbig-\rho}{\radbig} = \frac{\sin(\alpha-\beta)}{\sin(\alpha + \beta)}.
\label{singl}
\end{equation}
Moreover, $\sin(\alpha + \beta) = \varepsilon/(2 \rho)$ and
\gll
\sin(\alpha - \beta) = \sin \alpha \cos \beta - \cos \alpha \sin \beta = \frac{\delta}{\|z\|} \cdot \frac{\sqrt{\radbig^2-({\|z\|}/2)^2}}\radbig - \frac{\eps}{2{\|z\|}} \cdot \frac{{\|z\|}}{2\radbig}. 
\gle
Using previous identities together with \ref{singl} we get
\gll
\radbig - \rho = \frac{ 2\rho}{\eps} \cdot \left (\delta \sqrt{ \frac{\radbig^2}{\|z\|^2} - \frac{1}{4}} - \frac{\eps}{4} \right ) = \rho \cdot \left ( \frac{2\delta}{\eps^2} \sqrt{\frac{\radbig^2}{(\|z\|/\eps)^2}-\frac{\eps^2}{4}} - \frac{1}{2} \right ).
\gle
Solving the previous identity for $\rho$ together with $\|z\|^2 = \delta^2 + \eps^2/4$ leads to $\rho = {\radbig}\cdot(\frac{1}{2} + 2D(\eps))^{-1}$
where 
\gll
D(\eps) = \frac{\delta}{\varepsilon^2} \cdot \sqrt{\frac{\radbig^2}{ \left (\frac{\delta}{\varepsilon}  \right )^2 + \frac{1}{4} }-\frac{\varepsilon^2}{4}}.
\gle
For arbitrary $\eta>0$ there exists $j_0 \in \N$ such that for all $j\geq j_0$ it holds that
\gll
D(\eps_j) \geq \frac{\radbig}{4\radsmall}- \frac{\eta}{4}.
\gle
This follows by straightforward computation using the elementary estimate $\sqrt{\tau} - \sqrt{\sigma} \leq \sqrt{\tau-\sigma}$ for $0\leq \sigma \leq \tau$. 
Hence, $\rho \leq \frac{2\radbig}{1+\radbig/\radsmall-\eta}$ and therefore any short arc of radius $\frac{2\radbig}{1+\radbig/\radsmall-\eta}$ joining $x$ and $y$ is contained in $\Omega$ if $\|x-y\|=\eps_j$. By Proposition \ref{sequencearcprop} and Lemma \ref{rconv:sequence} this implies the claim in the limit $\eta \to 0$.
By Lemma \ref{Lemma:delta:circ}, $\Omega$ is $\radbig_0$-convex where $\radbig_0=2r$. So inductively, $\Omega$ is $\radbig_{n+1}$-convex for $\radbig_{n+1} = \frac{2\radbig_{n}}{1+\radbig_{n}/r}$ and any $n \in \N$. Now $\Omega$ is $r$-convex by Lemma \ref{rconv:sequence} as $\lim_{n \to \infty} \radbig_n = r$.
\end{proof}
Next, the limit inferior in \ref{limgl} will be replaced by the limit in order to verify Theorem \ref{existlim}.
\begin{proof}[Proof of Theorem \ref{existlim}]
Without loss of generality, let $\Omega$ consist of more than one point. Assume the nonexistence of the limit, that is, $\kappa <\limsup_{\eps>0, \varepsilon \to 0} \Quot{\varepsilon}$ for $\kappa := \liminf_{\eps>0, \varepsilon\to 0} \Quot{\varepsilon}$. Then there exists a sequence $(\varepsilon_i)_{i \in \N}$ converging to 0 such that 
\gll
\kappa < \lim_{i \to \infty} \quot{\varepsilon_i}.
\gle
It follows that there exists $s>0$ such that $\Quot{\eps_i} \geq \fracquot{8s}> \kappa$ for sufficiently large $i$. For any $s_0 > s$ and sufficiently large $i$, we have 
\gll
\inf_{x \in \Omega} \delta^\circ_{\Omega,x}(\varepsilon_i)\geq \delta_\Omega(\eps_i)\geq \frac{\eps_i^2}{8s} \geq s_0 - \sqrt{s_0^2-\eps_i^2/4}.
\gle
So property \ref{eq:sequence:circle} of Proposition \ref{thm:delta:circ} is fulfilled for any radius $s_0 >s$, so $\Omega$ is $s$-convex. Then $\kappa \geq 1/(8s)$, which is a contradiction. The limit is indeed finite: Suppose the limit equals $\infty$. Then $\Omega$ is $r$-convex for any $r>0$ by Proposition \ref{strongconvexityliminf}. Hence, by Proposition \ref{p:globalchar}, either $\Omega$ is a singleton or $\Omega = X$, which is a contradiction.
\end{proof}
We note that once the existence of the limit has been established, its finiteness is also implied by the following result from \cite{BalashovRepovs09}, which holds in the more general setting of a uniformly convex Banach space $X$: If $\Omega \subseteq X$ is closed and uniformly convex, and $\Omega \neq X$, then $\delta_\Omega(\varepsilon) \leq C \cdot \varepsilon^2$ for some $C \in \R$ and all sufficiently small $\varepsilon>0$.
\begin{proof}[Proof of Theorem \ref{localmainthm}]
The equivalence of \ref{localmainthm:i} and \ref{localmainthm:ii} follows from Proposition \ref{strongconvexityliminf}. Left to prove is \ref{localmainthm:iv} $\fol$ \ref{localmainthm:i} since \ref{localmainthm:ii} $\fol$  \ref{localmainthm:iv} is obvious. Let $x \in \Omega$ and $s> r$. By hypothesis it holds that $\delta_{\Omega,x}^\circ(\eps) \geq s-\sqrt{s^2-\eps^2/4}$ for all $\eps>0$ sufficiently small. From Proposition \ref{thm:delta:circ} we conclude that $\Omega$ is $s$-convex, and hence the assertion follows from Lemma \ref{rconv:sequence}.
\end{proof}
\appendix
\section*{Acknowledgment}
The authors thank M.V. Balashov (Moscow), J. Gwinner (Neubiberg) and A. Hinrichs (Jena) for their valuable comments on an earlier version of this paper.
%
%
%
%
\bibliographystyle{elsarticle-harv}
\bibliography{preambles,mrabbrev,strings,fremde,eigeneJOURNALS,stronglyconv12}
\end{document}